\documentclass[a4paper]{article}

\usepackage{hyperref} 

\usepackage{amssymb,amsmath,amsthm}
\usepackage[english]{babel}

\usepackage{todonotes}
\usepackage{verbatim} 
\usepackage{enumerate}
\usepackage{mathtools}

\usepackage{bbm}

\usepackage{color}
\definecolor{red}{RGB}{255,0,0}

\usepackage{cite}

\newcommand{\cB}{\mathcal{B}}

\newcommand{\cD}{\mathcal{D}}

\newcommand{\cK}{\mathcal{K}}
\newcommand{\cL}{\mathcal{L}}

\newcommand{\cN}{\mathcal{N}}

\newcommand{\cP}{\mathcal{P}}

\newcommand{\cT}{\mathcal{T}}

\newcommand{\fH}{\mathfrak{H}}

\newcommand{\fS}{\mathfrak{S}}

\newcommand{\bC}{\mathbb{C}}

\newcommand{\bE}{\mathbb{E}}

\newcommand{\bN}{\mathbb{N}}

\newcommand{\bR}{\mathbb{R}}

\newcommand{\PR}{\mathbb{P}}

\newcommand{\bONE}{\mathbbm{1}}

\newcommand{\dd}{\, \mathrm{d}}
\newcommand{\re}{ \mathrm{Re}\,}

\DeclareMathOperator{\Tr}{Tr} 
 
\renewcommand{\epsilon}{\varepsilon}

\newcommand{\vn}[1]{\left| \! \left| #1\right| \! \right|} 
\newcommand{\tn}[1]{\left| \! \left| \! \left|#1\right| \! \right| \! \right|}

\newcommand{\ip}[2]{\langle #1,#2\rangle}

\newcommand{\p}[1]{p\left( #1\right)}

\numberwithin{equation}{section}

\newtheorem{theorem}{Theorem}[section]
\newtheorem{lemma}[theorem]{Lemma}
\newtheorem{proposition}[theorem]{Proposition}
\newtheorem{corollary}[theorem]{Corollary}

\theoremstyle{definition}
\newtheorem{definition}[theorem]{Definition}

\newtheorem{remark}[theorem]{Remark}

\newtheorem{example}[theorem]{Example}
\newtheorem{condition}[theorem]{Condition}

\newtheorem*{Condition}{Condition}

\begin{document}

\title{Strongly continuous and locally equi-continuous semigroups on locally convex spaces}

\author{
\renewcommand{\thefootnote}{\arabic{footnote}}
Richard Kraaij
\footnotemark[1]
}

\footnotetext[1]{
Delft Institute of Applied Mathematics, Delft University of Technology, Mekelweg 4, 
2628 CD Delft, The Netherlands, E-mail: \texttt{r.c.kraaij@tudelft.nl}.
}

\date{20-04-2014}

\maketitle

\begin{abstract}
We consider locally equi-continuous strongly continuous semigroups on locally convex spaces $(X,\tau)$ that are also equipped with a `suitable' auxiliary norm. We introduce the set $\cN$ of $\tau$-continuous semi-norms that are bounded by the norm. If $(X,\tau)$ has the property that $\cN$ is closed under countable convex combinations, then a number of Banach space results can be generalised in a straightforward way. Importantly, we extend the Hille-Yosida theorem.

We relate our results to those on bi-continuous semigroups and show that they can be applied to semigroups on $(C_b(E),\beta)$ and $(\cB(\fH),\beta)$ for a Polish space $E$ and a Hilbert space $\fH$ and where $\beta$ is their respective strict topology.
\end{abstract}



\section{Introduction}

The study of Markov processes on complete separable metric spaces $(E,d)$ naturally leads to transition semigroups on $C_b(E)$ that are not strongly continuous with respect to the norm. Often, these semigroups turn out to be strongly continuous with respect to the weaker locally convex strict topology.

This leads to the study of strongly continuous semigroups on locally convex spaces. For equi-continuous semigroups, the theory is developed analogously to the Banach space situation for example in Yosida\cite{Yo78}. When characterising the operators that generate a semigroup, the more general context of locally equi-continuous semigroups introduces new technical challenges. Notably, the integral representation of the resolvent is not necessarily available. To solve this problem, K{\=o}mura, \={O}uchi and Dembart \cite{Ko68,Ou73,De74} have studied various generalised resolvents. More recently, Albanese and K\"{u}hnemund \cite{AK02} also study asymptotic pseudo resolvents and give a Trotter-Kato approximation result and the Lie-Trotter product formula.

A different approach is used in recent papers where a subclass of locally convex spaces $(X,\tau)$ is considered for which the ordinary representation of the resolvent can be obtained. Essentially, these spaces are also equipped with a norm $\vn{\cdot}$ such $(X,\vn{\cdot})$ is Banach and such that the dual $(X,\tau)'$ is norming for $(X,\vn{\cdot})$. Bi-continuous semigroups have been studied in \cite{Ku03,AM04,Fa04}, in which the Hille-Yosida, Trotter Approximation theorem and perturbation results have been shown. Bi-continuity has the drawback, however, that it is a non-topological notion.
Kunze \cite{Ku09,Ku11} studies semigroups of which he assumes that the resolvent can be given in integral form. His notions are topological, and he gives a Hille-Yosida theorem for equi-continuous semigroups.

\smallskip

In Section \ref{section:first_set}, we start with some minor results for locally convex spaces $(X,\tau)$ that are strong Mackey. These spaces are of interest, because a strongly continuous semigroup on a strong Mackey space is automatically locally equi-continuous, which extends a result by K{\=o}mura\cite{Ko68} for barrelled spaces.

From that point onward, we will consider sequentially complete locally convex spaces $(X,\tau)$ that are additionally equipped with an `auxiliary' norm. We assume that the norm topology is finer than $\tau$, but that the norm and $\tau$ bounded sets coincide. In Section \ref{section:boundedness_conditions}, we define $\cN$ as the set of $\tau$ continuous semi-norms that are bounded by the norm. We say that the space satisfies Convexity Condition C if $\cN$ is closed under taking countable convex combinations. This property allows the generalisation of a number of results in the Banach space theory. First of all, strong continuity of a semigroup on a space satisfying Condition C implies the exponential boundedness of the semigroup. Second, in Section \ref{section:infinitesimal_properties}, we show that the resolvent can be expressed in integral form. Third, in Section \ref{section:generation_results}, we give a straightforward proof of the Hille-Yosida theorem for strongly continuous and locally equi-continuous semigroups.

The strength of spaces that satisfy Condition C and the set $\cN$ is that results from the Banach space theory generalise by replacing the norm by semi-norms from $\cN$. Technical difficulties arising from working with the set $\cN$ instead of the norm are overcome by the probabilistic techniques of stochastic domination and Chernoff's bound, see Appendix \ref{section:appendix_orderings}.

\smallskip

In Section \ref{section:comparison_bicontinuous}, we consider $\tau$ bi-continuous semigroups. We show that if the so called mixed topology $\gamma = \gamma(\vn{\cdot},\tau)$, introduced by Wiweger\cite{Wi61}, has good sequential properties, bi-continuity of a semigroup for $\tau$ is equivalent to strong continuity and local equi-continuity for $\gamma$. 

In Section \ref{section:application_markov}, we show that the spaces $(C_b(E),\beta)$ and $(\cB(\fH),\beta)$, where $E$ is a Polish space, $\fH$ a Hilbert space and where $\beta$ is their respective strict topology, are strong Mackey and satisfy Condition C. This implies that our results can be applied to Markov transition semigroups on $C_b(E)$ and quantum dynamical semigroups on $\cB(\fH)$.

\section{Preliminaries} \label{section:general_definitions}

We start with some notation. Let $(X,\tau)$ be a locally convex space. We call the family of operators $\{T(t)\}_{t \geq 0}$ a \textit{semigroup} if $T(0) = \bONE$ and $T(t)T(s) = T(t+s)$ for $s,t \geq 0$. A family of $(X,\tau)$ continuous operators $\{T(t)\}_{t \geq 0}$ is called a \textit{strongly continuous semigroup} if $t \mapsto T(t)x$ is continuous and \textit{weakly continuous} if $t \mapsto \ip{T(t)x}{x'}$ is continuous for every $x \in X$ and $x \in X'$. 

We call $\{T(t)\}_{t \geq 0}$ a \textit{locally equi-continuous} family if for every $t \geq 0$ and continuous semi-norm $p$, there exists a continuous semi-norm $q$ such that $\sup_{s \leq t} p(T(s)x) \leq q(x)$ for every $x \in X$. 

Furthermore, we call $\{T(t)\}_{t \geq 0}$ a \textit{quasi equi-continuous} family if there exists $\omega \in \bR$ such that for every continuous semi-norm $p$, there exists a continuous semi-norm $q$ such that $\sup_{s \geq 0} e^{-\omega t} p(T(s)x) \leq q(x)$ for every $x \in X$.
Finally, we abbreviate strongly continuous and locally equi-continuous semigroup to \textit{SCLE} semigroup.

\smallskip

We use the following notation for duals and topologies. $X^*$ is the algebraic dual of $X$ and $X'$ is the continuous dual of $(X,\tau)$. Finally, $X^+$ is the sequential dual of $X$:
\begin{equation*}
X^+ := \{f \in X^* \, | \, f(x_n) \rightarrow 0, \text{ for every sequence } x_n \in X \text{ converging to 0} \}.
\end{equation*}
We write $(X,\sigma(X,X'))$, $(X,\mu(X,X'))$, $(X,\beta(X,X'))$, for $X$ equipped with the weak, Mackey or strong topology. Similarly, we define the weak, Mackey and strong topologies on $X'$. For any topology $\tau$, we use $\tau^+$ to denote the strongest locally convex topology having the same convergent sequences as $\tau$ \cite{We68}.

\section{Strong Mackey spaces: Connecting strong continuity and local equi-continuity} \label{section:first_set}

We start with a small exposition on a subclass of locally convex spaces that imply nice `local' properties of semigroups. 

K{\=o}mura\cite[Proposition 1.1]{Ko68} showed that on a barrelled space a strongly continuous semigroup is automatically locally equi-continuous. This fact is proven for the smaller class of Banach spaces in Engel and Nagel \cite[Proposition I.5.3]{EN00}, where they use the strong continuity of $\{T(t)\}_{t \geq 0}$ at $t = 0$ and the Banach Steinhaus theorem.

This approach disregards the fact that $\{T(t)\}_{t \geq 0}$ is strongly continuous for all $t \geq 0$ and Kunze \cite[Lemma 3.8]{Ku09} used this property to show that, in the case that every weakly compact subset of the dual is equi-continuous, strong continuity implies local equi-continuity.

\begin{definition}
We say that a locally convex space $(X,\tau)$ is \textit{strong Mackey} if all $\sigma(X',X)$ compact sets in $X'$ are equi-continuous.
\end{definition}

Following the proof of Lemma 3.8 in \cite{Ku09}, we obtain the following result.

\begin{lemma} \label{lemma:strongly_cont_implies_equicont}
If a semigroup $\{T(t)\}_{t\geq 0}$ of continuous operators on a strong Mackey space is strongly continuous, then the semigroup is locally equi-continuous.
\end{lemma}

We start with a proposition that gives sufficient conditions for a space to be strong Mackey.

\begin{proposition} \label{proposition:what_implies_compactequicont}
Any of the following properties implies that $(X,\tau)$ is strong Mackey.
\begin{enumerate}[(a)]
\item $(X,\tau)$ is barrelled.
\item $(X,\tau)$ is sequentially complete and bornological.
\item The space $(X,\tau)$ is sequentially complete, Mackey and the continuous dual $X'$ of $X$ is equal to the sequential dual $X^+$ of $X$.
\end{enumerate}
\end{proposition}
A space for which $X^+ = X'$ is called a \textit{Mazur} space\cite{Wi81}, or \textit{weakly semi bornological}\cite{BS96}. Note that a Mackey Mazur space satisfies $\tau = \tau^+$ by Corollary 7.6 in \cite{Wi81}. On the other hand, a space such that $\tau^+ = \tau$ is Mazur.

\begin{proof}
By K\"{o}the\cite[21.2.(2)]{Ko69}, the topology of a barrelled space coincides with the strong topology $\beta(X,X')$, in other words, all weakly bounded, and thus all weakly compact, sets are equi-continuous. 

Statement (b) follows from (a) as a sequentially complete bornological space is barrelled, see 28.1.(2) in K\"{o}the\cite{Ko69}. 

We now prove (c). The sequential completeness of $(X,\tau)$ and $X' = X^+$ imply that $(X',\mu(X',X))$ is complete by Corollary 3.6 in Webb \cite{We68}.

Let $K \subseteq X'$ be $\sigma(X',X)$ compact. By Krein's theorem \cite[24.4.(4)]{Ko69}, the completeness of $(X',\mu(X',X))$ implies that the absolutely convex cover of $K$ is also $\sigma(X',X)$ compact. By the fact that $\tau$ is the Mackey topology, every absolutely convex compact set in $(X',\sigma(X',X))$ is equi-continuous \cite[21.4.(1)]{Ko69}. This implies that $K$ is also equi-continuous.
\end{proof}

As an application of Lemma \ref{lemma:strongly_cont_implies_equicont}, we have the following proposition, which states that strong continuity is determined by local properties of the semigroup.

\begin{proposition} \label{proposition:equivalences_strong_continuity}
A semigroup $\{T(t)\}_{t\geq 0}$ of continuous operators on a strong Mackey space is strongly continuous if and only if the following two statements hold
\begin{enumerate}[(i)]
\item There is a dense subset $D \subseteq X$ such that $\lim_{t \rightarrow 0} T(t)x = x$ for every $x \in D$.
\item $\{T(t)\}_{t \geq 0}$ is locally equi-continuous.
\end{enumerate}
\end{proposition}

In the Banach space setting, strong continuity of the semigroup is equivalent to strong continuity at $t=0$, see Proposition I.5.3 in Engel and Nagel \cite{EN00}. In the more general situation, this equivalence does not hold, see Example 5.2 in Kunze \cite{Ku09}.

\begin{proof}
Suppose that $\{T(t)\}_{t \geq 0}$ is strongly continuous.
(i) follows immediately and (ii) follows from Lemma \ref{lemma:strongly_cont_implies_equicont}.

\smallskip

For the converse, suppose that we have (i) and (ii) for the semigroup $\{T(t)\}_{t \geq 0}$. First, we show that $\lim_{t \downarrow 0} T(t)x = x$ for every $x \in X$. Pick some $x \in X$ and let $x_\alpha$ be an approximating net in $D$ and let $p$ be a continuous semi-norm and fix $\varepsilon > 0$. We have
\begin{equation*}
p(T(t)x - x)  \leq p(T(t)x - T(t)x_\alpha) + p(T(t)x_\alpha - x_\alpha) + p(x_\alpha - x). 
\end{equation*}
Choose $\alpha$ large enough such that the first and third term are smaller than $\varepsilon/3$. This can be done independently of $t$, for $t$ in compact intervals, by the local equi-continuity of $\{T(t)\}_{t \geq 0}$. Now let $t$ be small enough such that the middle term is smaller than $\varepsilon/3$.

\smallskip

We proceed with the proving the strong continuity of $\{T(t)\}_{t \geq 0}$. The previous result clearly gives us $\lim_{s \downarrow t} T(s)x = T(t)x$ for every $x \in X$, so we are left to show that $\lim_{s \uparrow t} T(s)x = T(t)x$.

For $h > 0$ and $x \in X$, we have $T(t-h)x - T(t)x = T(t-h)\left(x - T(h)x\right)$, so the result follows by the right strong continuity and the local equi-continuity of the semigroup $\{T(t)\}_{t \geq 0}$.
\end{proof}

A second consequence of Lemma \ref{lemma:strongly_cont_implies_equicont}, for quasi complete spaces, follows from Proposition 1 in Albanese, Bonet and Ricker \cite{ABR12}.

\begin{proposition} \label{prop:stronglycont_iff_weaklycont}
Suppose that we have a semigroup of continuous operators $\{T(t)\}_{t \geq 0}$ on a quasi complete strong Mackey space. Then the semigroup is strongly continuous if and only if it is weakly continuous and locally equi-continuous.
\end{proposition}

As in the Banach space situation, it would be nice to have some condition that implies that the semigroup, suitably rescaled is globally bounded. We directly run into major restrictions.

\begin{example}
Consider $C_c^\infty(\bR)$ the space of test functions, equipped with its topology as a countable strict inductive limit of Fr\'{e}chet spaces. 
This space is complete \cite[Theorem 13.1]{Tr67}, Mackey \cite[Propositions 34.4 and 36.6]{Tr67} and $C_c^\infty(\bR)^+ = C_c^\infty(\bR)'$ as a consequence of \cite[Corollary 13.1.1]{Tr67}.

Define the semigroup $\{T(t)\}_{t \geq 0}$ by setting $\left(T(t)f\right)(s) = f(t+s)$. This semigroup is strongly continuous, however, even if exponentially rescaled, it can never be globally bounded by 19.4.(4) \cite{Ko69}.
\end{example}

So even if $(X,\tau)$ is strong Mackey, we can have semigroups that have undesirable properties. This issue is serious. For example, in the above example, formally writing the resolvent corresponding to the semigroup in its integral form, yields a function which is not in $C_c^\infty(\bR)$. One can work around this problem, see for example \cite{De74,Ko68,Ou73} which were already mentioned in the introduction.

However, motivated by the study of Markov processes, where the resolvent informally corresponds to evaluating the semigroup at an exponential random time, we would like to work in a framework in which the ordinary integral representation for the resolvent holds.

\section{A suitable structure of bounded sets} \label{section:boundedness_conditions}

In this section, we shift our attention to another type of locally convex spaces. As a first major consequence, we are able to show in Corollary \ref{corollary:exponential_continuity} an analogue of the exponential boundedness of a strongly continuous semigroup on a Banach space. This indicates that we may be able to mimic major parts of the Banach space theory.

Suppose that $(X,\tau)$ is a locally convex space, and suppose that $X$ can be equipped with a norm $\vn{\cdot}$, such that $\tau$ is weaker than the norm topology. It follows that bounded sets for the norm are bounded sets for $\tau$. This means that if we have a $\tau$-continuous semi-norm $p$, then there exists some $M > 0$ such that $\sup_{x : \vn{x} \leq 1} p(x) \leq M$. Therefore, $p(x) \leq M \vn{x}$ for every $x$, i.e. every $\tau$-continuous semi-norm is dominated by a constant times the norm.

\begin{definition} \label{definition:def_of_N}
Let $(X,\tau)$ be equipped with a norm $\vn{\cdot}$ such that $\tau$ is weaker than the norm topology. Denote by $\cN$ the $\tau$-continuous semi-norms that satisfy $p(\cdot) \leq \vn{\cdot}$. We say that $\cN$ is \textit{countably convex} if for any sequence $p_n$ of semi-norms in $\cN$ and $\alpha_n \geq 0$ such that $\sum_n \alpha_n = 1$, we have that $p(\cdot) := \sum_n \alpha_n p_n(\cdot) \in \cN$.
\end{definition}

We start with exploring the situation where $\tau$ and $\vn{\cdot}$ have the same bounded sets.

\begin{Condition}[Boundedness condition B]
A locally convex space $(X,\tau)$ also equipped with a norm $\vn{\cdot}$, denoted by $(X,\tau,\vn{\cdot})$, satisfies \textit{Condition B} if 
\begin{enumerate}[(a)]
\item $\tau$ is weaker than the norm topology.
\item Both topologies have the same bounded sets.
\end{enumerate}
\end{Condition}

\begin{remark} \label{remark:mixed_topology}
Suppose that $(X,\tau)$ is a locally convex space, and suppose that $\vn{\cdot}$ is a norm on $X$ such that the norm topology is stronger than $\tau$, but such that the norm topology has less bounded sets than $\tau$. 

In this case, it is useful to consider the \textit{mixed} topology $\gamma = \gamma(\vn{\cdot},\tau)$, introduced in Wiweger \cite{Wi61}. In Section \ref{section:comparison_bicontinuous}, we study the relation of bi-continuous semigroups for $\tau$ with SCLE semigroups for $\gamma$.
\end{remark}

We introduce some notation. We write $X'_n := (X,\vn{\cdot})'$ and $X'_\tau := (X,\tau)'$. Also, we denote $B_n := \{x' \in X_n' \, | \, \vn{x'}' \leq 1 \}$, where $\vn{\cdot}'$ is the operator norm on $X_n'$. Finally, we set $B_\tau = B_n \cap X_\tau'$. We start with a well known theorem that will aid our exposition. 

\begin{theorem}[Bipolar Theorem] \label{theorem:bipolar}
Let $(X,\tau)$ be a locally convex space and let $\vn{\cdot}$ be a norm on $X$. Let $p$ be a $\tau$ lower semi-continuous semi-norm such that $p \leq \vn{\cdot}$. Then there exists a absolutely convex weakly bounded set $\fS := \{p \leq 1\}^\circ \subseteq B_\tau$ such that
\begin{equation*}
p(x) = \sup_{x' \in \fS} |\ip{x}{x'}|.
\end{equation*}
Furthermore, $p$ is continuous if and only if $\fS$ is an equi-continuous set.
\end{theorem}

\begin{proof}
The result follows from 20.8.(5) and 21.3.(1) in K\"{o}the\cite{Ko69}. The fact that $\fS \subseteq B_\tau$ is a consequence of $p \in \cN$.
\end{proof}

\begin{lemma} \label{lemma:equivalences_boundedness}
Let $(X,\tau)$ be sequentially complete locally convex space, and $\vn{\cdot}$ a norm on $X$ such that the norm topology is stronger than $\tau$. Then the following are equivalent.
\begin{enumerate}[(a)]
\item The norm bounded sets equal the $\tau$ bounded sets.
\item $\vn{\cdot}$ is $\tau$ lower semi-continuous.
\item The norm can be expressed as $\vn{x} = \sup_{x' \in B_\tau} |\ip{x}{x'}|$.
\end{enumerate}
In all cases, the topology generated by $\vn{\cdot}$ is the $\beta(X,X_\tau')$ topology and is Banach. The norm can equivalently be written as 
\begin{equation} \label{eqn:norm_given_as_sup_over_N}
\vn{x} = \sup_{p \in \cN} p(x).
\end{equation}
\end{lemma}

\begin{proof}
We start with the proof of (a) to (b). Define the $\beta(X,X_\tau')$ continuous norm $\tn{x} := \sup_{x' \in B_\tau} |\ip{x}{x'}|$. Note that $\tn{\cdot} \leq \vn{\cdot}$ by construction. It follows that the $\vn{\cdot}$ topology is stronger than the $\tn{\cdot}$ topology, which is in turn stronger than $\tau$. The bounded sets of the two extremal topologies are the same, so the $\vn{\cdot}$ and the $\tn{\cdot}$ bounded sets coincide. Thus, there is some $c \geq 1$ such that $\tn{\cdot} \leq \vn{\cdot} \leq c \tn{\cdot}$. But this means that $\vn{\cdot}$ is $\beta(X,X_\tau')$ continuous, and thus $\tau$ lower semi-continuous.

Now assume (b), we prove (a). As $\tau$ is weaker than the norm topology, it follows that the norm topology has less bounded sets. On the other hand, as the norm is $\tau$ lower semi-continuous, it is continuous for the strong topology $\beta(X,X_\tau')$. Therefore, the strong topology has less bounded sets than the norm topology. As $(X,\tau)$ is sequentially complete, the Banach-Mackey theorem, 20.11.(3) in K\"{o}the\cite{Ko69} shows that the strongly bounded sets and the $\tau$ bounded sets coincide, which implies (a).

(c) clearly implies (b) and (b) implies (c) by the Bipolar theorem. 

$(X,\tau)$ is Banach by 18.4.(4) in \cite{Ko69}. 
\end{proof}

The usefulness of $\cN$ becomes clear from the next three results. Intuitively, the next two lemmas tell us that in the study of semigroups on these locally convex spaces the collection $\cN$ replaces the role that the norm plays for semigroups on Banach spaces.

\begin{lemma} \label{lemma:equivalence_norm_and_N_bounds}
Let $(X,\tau,\vn{\cdot})$ satisfy Condition B. Let $I$ be some index set and let $(T_\alpha)_{\alpha \in I}$ be $(X,\tau)$ to $(X,\tau)$ continuous operators. Then the following are equivalent
\begin{enumerate}[(a)]
\item The family $\{T_\alpha\}_{\alpha \in I}$ is $\tau$-equi-continuous and $\sup_{\alpha \in I} \vn{T_\alpha} \leq M$.
\item For every $p \in \cN$, there is $q \in \cN$ such that $\sup_{\alpha \in I} p(T_\alpha x) \leq M q(x)$ for all $x \in X$.
\end{enumerate}
Furthermore, if the family $\{T_\alpha\}_{\alpha \in I}$ is $\tau$-equi-continuous, then there exists $M \geq 0$ such that these properties hold.
\end{lemma}

\begin{proof}
The implication (b) to (a) follows from Equation \eqref{eqn:norm_given_as_sup_over_N}. For the proof of (a) to (b), fix some semi-norm $p \in \cN$. As the family $\{T_\alpha\}_{\alpha \in I}$ is $\tau$-equi-continuous, there is some continuous semi-norm $\hat{q}$ such that $\sup_{\alpha \in I} p(T_\alpha x) \leq \hat{q}(x)$. This implies that $q(x) := M^{-1} \sup_{\alpha \in I}  p(T_\alpha x)$ is $\tau$-continuous. We conclude that $q \in \cN$ by noting that
\begin{equation*}
q(x) = \frac{1}{M} \sup_{\alpha \in I} p(T_\alpha x) \leq \frac{1}{M} \sup_{\alpha \in I} \vn{T_\alpha x} \leq \vn{x}.
\end{equation*}
If the family $\{T_\alpha\}_{\alpha \in I}$ is $\tau$-equi-continuous, it is $\tau$-equi-bounded which implies that there is some $M \geq 0$ such that $\sup_{\alpha \in I} \vn{T_\alpha} \leq M$ by Condition B.
\end{proof}

In particular, we have the following result.

\begin{lemma} \label{lemma:equivalence_equicontinuity}
Let $(X,\tau,\vn{\cdot})$ satisfy Condition B and $\{T(t)\}_{t \geq 0}$ be a semigroup of continuous operators. Then the following are equivalent.
\begin{enumerate}[(a)]
\item $\{T(t)\}_{t \geq 0}$ is locally equi-continuous.
\item For every $t \geq 0$ there exists $M \geq 1$, such that for every $p \in \cN$ there exists $q \in \cN$ such that for all $x \in X$
\begin{equation*}
\sup_{s \leq t} p(T(s)x) \leq M q(x).
\end{equation*}
\end{enumerate}
\end{lemma}

As a corollary, we obtain an exponential growth bound.

\begin{corollary} \label{corollary:exponential_continuity}
Let $(X,\tau,\vn{\cdot})$ satisfy Condition B. For a locally equi-continuous semigroup $\{T(t)\}_{t \geq 0}$, there is $M \geq 1$ and $\omega \in \bR$ such that for every $T \geq 0$ and every $p \in \cN$ there is a $q \in \cN$ such that for all $x \in X$
\begin{equation*}
\sup_{t \leq T} e^{-\omega t} p(T(t)x) \leq M q(x).
\end{equation*}
\end{corollary}

\begin{proof}
Pick $M \geq 1$ such that for every $p \in \cN$ there exists $q \in \cN$ such that
\begin{equation} \label{eqn:dominationseminorms}
\sup_{t \leq 1} p(T(t)x) \leq M q(x)
\end{equation}
for every $x \in X$. Without loss of generality, we can always choose $q \in \cN$ to dominate $p$. We use this property to construct an increasing sequence of semi-norms in $\cN$.

Fix some $p \in \cN$ and pick $q_0 \geq p$ such that it satisfies the property in Equation \eqref{eqn:dominationseminorms}. Inductively, let $q_{n+1} \in \cN$ be a semi-norm such that $q_{n+1} \geq q_n$ and $\sup_{t \leq 1} q_{n+1}(T(t)x) \leq M q_n(x)$.
Now let $t \geq 0$. Express $t = s +n$ where $n \in \bN$ and $0 \leq s < 1$, then it follows that
\begin{equation*}
p(T(t)x) \leq M q_0(T(n)) \leq \cdots \leq M^{n+1} q_n(x) \leq M e^{t \log M} q_n(x).
\end{equation*}
Setting $\omega = \log M$, we obtain $\sup_{t \leq T} e^{-\omega t} p(T(t)x) \leq M q_{\lceil T \rceil}(x)$ for every $x \in X$.
\end{proof}

This last result inspires the following definition, which is clearly analogous to the situation for semigroups in Banach spaces.

\begin{definition}
We say that a semigroup on a space $(X,\tau,\vn{\cdot})$ that satisfies Condition B is of type $(M,\omega)$, $M \geq 1$ and $\omega \in \bR$, if for every $p \in \cN$ and $T \geq 0$ there exists $q \in \cN$ such that
\begin{equation*}
\sup_{t \leq T} e^{-\omega t} p(T(t)x) \leq M q(x)
\end{equation*}
for all $x \in X$. We say that it is of type $(M,\omega)^*$ if
\begin{equation*}
\sup_{t \geq 0} e^{-\omega t} p(T(t)x) \leq M q(x).
\end{equation*}

Furthermore, we define the growth bound $\omega_0$ of $\{T(t)\}_{t \geq 0}$ by
\begin{equation*}
\omega_0 := \inf\left\{\omega \in \bR \, \middle| \, \exists M \geq 1 \text{ such that } \{T(t)\}_{t \geq 0} \text{ is of type } (M,\omega) \right\}.
\end{equation*}
\end{definition}

It follows that if a semigroup is of type $(M,\omega)$ for some $M$ and $\omega$, then it is locally equi-continuous. Furthermore, if it is of type $(M,\omega)^*$ it is quasi equi-continuous.

\begin{Condition}[Convexity condition C]
A locally convex space $(X,\tau)$ also equipped with a norm $\vn{\cdot}$, denoted by $(X,\tau,\vn{\cdot})$, satisfies \textit{Condition C} if 
\begin{enumerate}[(a)]
\item $(X,\tau)$ is sequentially complete.
\item $\tau$ is weaker than the norm topology.
\item Both topologies have the same bounded sets.
\item $\cN$ is countably convex.
\end{enumerate}
\end{Condition}

We give some conditions that imply that $\cN$ is countably convex. Interestingly, the same spaces that are strong Mackey, if equipped with a suitable norm, also turn out to satisfy Condition C. 

We say that a space $(X,\tau)$ is \textit{transseparable} if for every open neighbourhood $U$ of $0$, there is a countable subset $A \subseteq X$ such that $A + U = X$. Note that a separable space is transseparable.

\begin{proposition} \label{prop:what_implies_type_C}
Let $(X,\tau)$ be a sequentially complete locally convex space that is also equipped with some norm $\vn{\cdot}$ such that $\tau$ is weaker than the norm topology and such that both topologies have the same bounded sets. The set $\cN$ is countably convex if either of the following hold
\begin{enumerate}[(a)]
\item $\tau^+ = \tau$.
\item $(X,\tau)$ is Mackey and $(X_\tau',\sigma(X_\tau',X))$ is locally complete. 
\item $(X,\tau)$ is transseparable and $(X_\tau',\sigma(X_\tau',X))$ is sequentially complete.
\end{enumerate}
Furthermore, (b) holds for all three classes of spaces mentioned in Proposition \ref{proposition:what_implies_compactequicont}.
\end{proposition}

Note that $\tau^+ = \tau$ is satisfied if $\tau$ is sequential. This holds for example if $(X,\tau)$ is Banach or Fr\'{e}chet. Local completeness of $(X_\tau',\sigma(X_\tau',X))$ is implied by sequential completeness of $(X_\tau',\sigma(X_\tau',X))$ \cite[Corollary 5.1.8]{PCB87}. If $(X_\tau',\sigma(X_\tau',X))$ is locally complete, then $(X,\tau)$ is called \textit{dual locally complete}\cite{SS97}.

\begin{proof}[Proof of Proposition \ref{prop:what_implies_type_C}]
Pick $p_n \in \cN$ and $\alpha_n \geq 0$, such that $\sum_n \alpha_n = 1$. Define $p(\cdot) = \sum_n \alpha_n p_n(\cdot)$. First of all, it is clear that $p$ is a semi-norm. Thus, we need to show that $p$ is $\tau$ continuous. 

\smallskip

Suppose that $\tau^+ = \tau$. By Theorem 7.4 in Wilansky\cite{Wi81} a sequentially continuous semi-norm is continuous. Thus it is enough to show sequential continuity of $p$. This follows directly from the dominated convergence theorem, as every $p_n$ is continuous and $p_n(\cdot) \leq \vn{\cdot}$.

\smallskip

For the proof of  (b) and (c), we need the explicit form of the semi-norms in $\cN$ given in Theorem \ref{theorem:bipolar}. Recall that $B_\tau := \{x' \in (X,\tau)' \, | \, \vn{x'}' \leq 1\}$. If $q \in \cN$, then there is an absolutely convex closed and equi-continuous set $\fS \subseteq B$ such that 
\begin{equation*}
q(\cdot) = \sup_{x' \in \fS} |\ip{\cdot}{x'}|.
\end{equation*}

We proceed with the proof of (b). The sequence of semi-norms $p_n$ are all of the type described above. So let $\fS_n$ be the equi-continuous subset of  $B_\tau$ that corresponds to $p_n$. Define the set
\begin{equation*}
\fS := \left\{\lim_{n \rightarrow \infty} \sum_{i=1}^n \alpha_i u_i \, \middle| \, u_i \in \fS_i\right\}.
\end{equation*}
The dual local completeness of $(E,\tau)$ shows that these limits exists by Theorem 2.3 in \cite{SS97}. Under the stronger assumption that $(X_\tau',\sigma(X_\tau',X))$ is sequentially complete this is obvious.

To finish the proof of case (b), we prove two statements. The first one is that $p(x) = \sup_{x' \in \fS} |\ip{x}{x'}|$, the second is that $\fS$ is $\tau$ equi-continuous. Together these statements imply that $p$ is $\tau$ continuous.

\smallskip

We start with the first statement. For every $x \in X$, there are $x'_n \in \fS_n$ such that $p_n(x) = \ip{x}{x'_n}$ by construction. Therefore,
\begin{equation*}
p(x) = \sum_{n=1}^\infty \alpha_n \ip{x}{x_n'}  = \ip{x}{\sum_{n=1}^\infty \alpha_n x_n'} = \sup_{x' \in \fS} |\ip{x}{x'}|.
\end{equation*}
On the other hand,
\begin{equation*}
\sup_{y' \in \fS} |\ip{x}{y'}| = \sup_{\substack{y_n' \in \fS_n \\ n \geq 1}} |\ip{x}{\sum_{n=1}^\infty \alpha_n y_n'}| \leq \sum_{n=1}^\infty \alpha_n \sup_{y_n' \in \fS_n} |\ip{x}{y_n'}| \leq p(x).
\end{equation*}
Combining these statements, we see that $p(x) = \sup_{x' \in \fS} |\ip{x}{x'}|$. 

\smallskip

We prove the equi-continuity of $\fS$. Consider $\fS_n$ equipped with the restriction of the $\sigma(X_\tau',X)$ topology. Define the product space $\cP := \prod_{n=1}^\infty \fS_n$ and equip it with the product topology. As every closed equi-continuous set is $\sigma(X_\tau',X)$ compact by the Bourbaki-Alaoglu theorem \cite[20.9.(4)]{Ko69}, $\cP$ is also compact.

Let $\phi : \cP \rightarrow \fS$ be the map defined by $\phi(\{x_n'\}_{n \geq 1}) = \sum_{n \geq 1} \alpha_n x_n'$.  Clearly, $\phi$ is surjective. We prove that $\phi$ is continuous. Let $\beta \mapsto \{x'_{\beta,n}\}_{n \geq 1}$ be a net converging to $\{x'_n\}_{n\geq 1}$ in $\cP$.  Fix $\varepsilon > 0$ and $f \in X$. Now let $N$ be large enough such that $\sum_{n > N} \alpha_n < \frac{1}{4 \vn{f}}\varepsilon$ and pick $\beta_0$ such that for every $\beta \geq \beta_0$ we have $\sum_{n \leq N} |\ip{f}{x'_{\beta,n} - x'_n}| \leq \frac{1}{2} \varepsilon$. Then, it follows for $\beta \geq \beta_0$ that
\begin{align*}
& \left|\phi(\{x'_{\beta,n}\}_{n \geq 1}) - \phi(\{x'_n\}_{n \geq 1})\right| \\
& \leq \sum_{n \leq N} \alpha_n|\ip{f}{x'_{\beta,n} - x'_n}| + \sum_{n > N} \alpha_n|\ip{f}{x'_{\beta,n} - x'_n}| \\
& \leq \frac{1}{2} \varepsilon + \sum_{n > N} \alpha_n \vn{f}\vn{x'_{\beta,n} - x'_n}' \\
& \leq \frac{1}{2} \varepsilon + 2\vn{f} \frac{1}{4\vn{f}} \varepsilon \\
& = \varepsilon,
\end{align*}
where we use in line four that all $x_{\beta,n}'$ and $x_n'$ are elements of $B_\tau$. As a consequence, $\fS$, as the continuous image of a compact set, is $\sigma(X_\tau',X)$ compact. $\fS$ is also absolutely convex, as it is the image under an affine map of an absolutely convex set. As $(X,\tau)$ is Mackey, this yields that $\fS$ is equi-continuous, which in turn implies that $p$ is $\tau$ continuous. 

\smallskip

The proof of (c) follows along the lines of the proof of (b). The proof changes slightly as we can not use that a $\sigma(X_\tau',X)$ compact set is equi-continuous. We replace this by using transseparability. We adapt the proof of (b).

As $(X,\tau)$ is transseparable, the $\sigma(X_\tau',X)$ topology restricted to $\fS_n$ is metrisable by Lemma 1 in Pfister\cite{Pf76}. This implies that the product space $\cP := \prod_{n=1}^\infty \fS_n$ with the product topology $\cT$ is metrisable.

By 34.11.(2) in K\"{o}the \cite{Ko79}, we obtain that $\fS$, as the continuous image of a metrisable compact set, is metrisable. The equi-continuity of $\fS$ now follows from corollaries of Kalton's closed graph theorem, see Theorem 2.4 and Theorem 2.6 in Kalton \cite{Ka71} or  34.11.(6) and 34.11.(9) in \cite{Ko79}.

\smallskip

We show that that the spaces mentioned in Proposition \ref{proposition:what_implies_compactequicont} satisfy (b). If $(X,\tau)$ is Mackey and $X' = X^+$, then $\tau^+ = \tau$ by Theorem 7.4 and Corollary 7.5 in \cite{Wi81}. A sequentially complete bornological space is barrelled, so to complete the proof, we only need to consider barrelled spaces. The topology of a barrelled space coincides with the strong topology, therefore a weak Cauchy sequence in $X_\tau'$ is equi-continuous and thus has a weak limit \cite[Proposition 32.4]{Tr67}.
\end{proof}

\section{Infinitesimal properties of semigroups} \label{section:infinitesimal_properties}

We now start with studying the infinitesimal properties of a semigroup. Besides the local equi-continuity which we assumed for all results in previous section, we will now also assume strong continuity.

We directly state the following weaker form of Proposition \ref{proposition:equivalences_strong_continuity} for later reference.

\begin{lemma} \label{lemma:equivalences_strong_continuity_2}
Let $\{T(t)\}_{t \geq 0}$ be a locally equi-continuous semigroup on a locally convex space $(X,\tau)$. Then the following are equivalent.
\begin{enumerate}[(a)]
\item $\{T(t)\}_{t \geq 0}$ is strongly continuous.
\item There is a dense subset $D \subseteq X$ such that $\lim_{t \downarrow 0} T(t)x = x$ for all $x \in X$.
\end{enumerate}
\end{lemma}

The \textit{generator} $(A,\cD(A))$ of a SCLE semigroup $\{T(t)\}_{t \geq 0}$ on a locally convex space $(X,\tau)$ is the linear operator defined by
\begin{equation*}
A x := \lim_{t \downarrow 0} \frac{T(t)x - x}{t}
\end{equation*}
for $x$ in the set
\begin{equation*}
\cD(A) := \left\{ x \in X \, \middle| \, \lim_{t \downarrow 0} \frac{T(t)x - x}{t} \text{ exists} \right\}.
\end{equation*}
We say that $(A,\cD(A))$ is \textit{closed} if $\{(x,Ax) \, | \, x \in \cD(A)\}$ is closed in the product space $X\times X$ with the product topology.

\smallskip

The generator $(A,\cD(A))$ satisfies the following well known properties. The proofs can be found for example as Propositions 1.2, 1.3 and 1.4 in K{\=o}mura\cite{Ko68}.

\begin{lemma} \label{lemma:EN.II.1.3}
Let $(X,\tau)$ be a locally convex space. For the generator $(A,\cD(A))$ of a SCLE semigroup $\{T(t)\}_{t \geq 0}$, we have 
\begin{enumerate}[(a)]
\item $\cD(A)$ is closed and dense in $X$.
\item For $x \in \cD(A)$, we have $T(t)x \in \cD(A)$ for every $t \geq 0$ and $\frac{\dd}{\dd t} T(t) x = T(t)A x = A T(t)x$.
\item For $x \in X$ and $t \geq 0$, we have $\int_0^t T(s)x \dd s \in \cD(A)$.
\item For $t \geq 0$, we have
\begin{align*}
T(t)x - x & = A \int_0^t T(s)x \dd s & & \text{if } x \in X \\
& = \int_0^t T(s)Ax \dd s & & \text{if } x \in \cD(A).
\end{align*}
\end{enumerate}
\end{lemma}

The integral in (d) should be understood as a $\tau$ Riemann integral. This is possible due to the strong continuity and the local-equi continuity of the semigroup.

Define the \textit{spectrum} of $(A,\cD(A))$ by $\sigma(A) := \{\lambda \in \bC \, | \, \lambda - A \text{ is not bijective}\}$, the \textit{resolvent set} $\rho(A) = \bC \setminus \sigma(A)$, for $\lambda \in \rho(A)$ the \textit{resolvent} $R(\lambda,A) = (\lambda - A)^{-1}$.

\begin{remark}
We will not touch on the subject in this paper, but sequential completeness implies that multiplication in the locally convex algebra of operators on $X$ is bounded. This is enough to develop spectral theory, for references see Section 40.5 in K\"{o}the \cite{Ko79}.
\end{remark}

\begin{proposition} \label{proposition:existence_resolvent}
Let $(X,\tau,\vn{\cdot})$ satisfy Condition C. Let $\{T(t)\}_{t \geq 0}$ be a SCLE semigroup with growth bound $\omega_0$.
\begin{enumerate}[(a)]
\item If $\lambda \in \bC$ is such that the improper Riemann-integral
\begin{equation*}
R(\lambda)x := \int_0^\infty e^{-\lambda t} T(t)x \dd t
\end{equation*}
exists for every $x \in X$, then $\lambda \in \rho(A)$ and $R(\lambda,A) = R(\lambda)$. 
\item Suppose that the semigroup is of type $(M,\omega)$. We have for every $\lambda \in \bC$ such that $\re \lambda > \omega$ and $x \in X$ that
\begin{equation*}
R(\lambda)x := \int_0^\infty e^{- \lambda t} T(t) x \, \dd t
\end{equation*}
exists as an improper Riemann integral. Furthermore, $\lambda \in \rho(A)$.
\item If $\re \lambda > \omega_0$, then $\lambda \in \rho(A)$.
\end{enumerate}
\end{proposition}

\begin{proof}
The proof of the first item is standard. We give the proof of (b) for completeness. Let $\lambda$ be such that $\re \lambda > \omega$. First, for every $a >0$ the integral $R_a(\lambda)x := \int_0^a e^{- \lambda t} T(t) x \dd t$ exists as a $\tau$ Riemann integral by the local equi-continuity of $\{T(t)\}_{t \geq 0}$ and the sequential completeness of $(X,\tau)$.

\smallskip

The sequence $n \mapsto R_n(\lambda)x$ is a $\tau$ Cauchy sequence for every $x \in X$, because for every semi-norm $p \in \cN$ and $m > n \in \bN$ there exists a semi-norm $q \in \cN$ such that
\begin{align*}
p\left(R_m(\lambda)x - R_n(\lambda)x \right) & \leq p\left(\int_n^m e^{-t \lambda} T(t) x \dd t\right) \\
& \leq p\left(\int_n^m e^{-t (\lambda - \omega)} e^{- \omega t} T(t) x \dd t\right) \\
& \leq M q(x) \int_n^m e^{- t (\re \lambda - \omega)}  \dd t \\
& \leq M \vn{x} \frac{e^{-\lambda m} - e^{- \lambda n}}{\re \lambda - \omega}.
\end{align*}
Therefore, $n \mapsto R_n(\lambda)x$ converges by the sequential completeness of $(X,\tau)$. (c) follows directly from (a) and (b).
\end{proof}

We have shown that if $\re \lambda > \omega_0$, then $\lambda \in \rho(A)$. We can say a lot more. 

\begin{theorem} \label{theorem:EN.II.1.10}
Let $(X,\tau,\vn{\cdot})$ satisfy Condition C. Let $\{T(t)\}_{t \geq 0}$ be a SCLE semigroup of growth bound $\omega_0$. For $\lambda  > \omega_0$, $R(\lambda,A)$ is a $\tau$ continuous linear map. Furthermore, if $\{T(t)\}_{t \geq 0}$ is of type $(M,\omega)$, then there exists for every $\lambda_0 > \omega_0$ and semi-norm $p \in \cN$ a semi-norm $q \in \cN$ such that
\begin{equation} \label{eqn:control_on_resolvents}
\sup_{\re \lambda \geq \lambda_0} \sup_{n \geq 0} \; (\re \lambda - \omega)^n p\left(\left(nR(n\lambda)\right)^n x\right) \leq M q(x)
\end{equation} 
for every $x \in X$. If $\{T(t)\}_{t \geq 0}$ is of type $(M,\omega)^*$, then the last statement can be strengthened to
\begin{equation*}
\sup_{\re \lambda > \omega} \sup_{n \geq 0} \; (\re \lambda - \omega)^n p\left(\left(nR(n\lambda)\right)^n x\right) \leq M q(x).
\end{equation*} 
\end{theorem}

For the proof of the theorem, we will make use of Chernoff's bound and the probabilistic concept of stochastic domination. A short explanation and some basic results are given in Appendix \ref{section:appendix_orderings}.

\begin{proof}
In the proof, we will write $\lceil s \rceil$ for the smallest integer $n \geq s$. Clearly, the $\tau$ continuity of $R(\lambda,A)$ follows directly from the result in Equation \eqref{eqn:control_on_resolvents}, so we will start to prove \eqref{eqn:control_on_resolvents}. Without loss of generality, we can rescale and prove the result for a semigroup of type $(M,0)$. 

\smallskip

Let $\lambda_0 > 0$. Fix some semi-norm $p \in \cN$. By the local equi-continuity of $\{T(t)\}_{t \geq 0}$, we can find semi-norms $q_n \in \cN$, increasing in $n$, such that $\sup_{s \leq n} p(T(s)x) \leq M q_n(x)$. 

By iterating the representation of the resolvent given in Proposition \ref{proposition:existence_resolvent}, we see
\begin{equation*}
\left(n \re \lambda  R(n\lambda)\right)^n x = \int_0^\infty \frac{(n\re \lambda)^n s^{n-1}}{(n-1)!} e^{- s n \lambda} T(s)x \dd s,
\end{equation*}
which implies
\begin{equation*}
p\left(\left(n \re \lambda  R(n\lambda)\right)^n x\right) \leq M \int_0^\infty \frac{(n\re \lambda)^n s^{n-1}}{(n-1)!} e^{- sn \re \lambda} q_{\lceil s \rceil}(x) \dd s
\end{equation*}
for every $x \in X$. On the right hand side, we have the semi-norm 
\begin{equation*}
q_{n,\re \lambda} := \int_0^\infty \frac{(n\re \lambda)^n s^{n-1}}{(n-1)!}  e^{- sn \re \lambda} q_{\lceil s \rceil} \dd s 
\end{equation*}
in $\cN$ by the countable convexity of $\cN$ and the fact that we integrate with respect to a probability measure. We denote this measure on $[0,\infty)$ by 
\begin{equation*}
\mu_{n,\re \lambda}(\dd s) = \frac{(n\re \lambda)^n s^{n-1}}{(n-1)!}  e^{- sn \re \lambda} \dd s,
\end{equation*}
and with $B_{n,\re \lambda}$ a random variable with this distribution. As a consequence, we have the following equivalent definitions: 
\begin{equation*}
q_{n,\re \lambda} = \int_0^\infty q_{\lceil s \rceil} \mu_{n, \re \lambda}(\dd s) = \bE\left[ q_{ \lceil B_{n,\re \lambda} \rceil} \right].
\end{equation*}

To show equi-continuity of $(n \re \lambda)^n R(\lambda n)^n$, we need to find one semi-norm $q \in \cN$ that dominates all $q_{n, \re \lambda}$ for $n \geq 0$ and $\re \lambda \geq \lambda_0$. Because $s \mapsto q_{\lceil s\rceil}(x)$ is an increasing and bounded function for every $x \in X$, the result follows from Lemma \ref{lemma:appendix_stochastic_domination_equivalence}, if we can find a random variable $Y$ that stochastically dominates all $B_{n, \re \lambda}$.

In other words, we need to find a random variable that dominates the tail of the distribution of all $B_{n, \re \lambda}$. To study the tails, we use Chernoff's bound, Proposition \ref{proposition:Chernoff}. 

\smallskip

Let $g(s,\alpha,\beta) := \frac{\beta^\alpha s^{\alpha - 1}}{\Gamma(\alpha)} e^{-\beta s}$, $s \geq 0$, $\alpha,\beta > 0$ be the density with respect to the Lebesgue measure of a $\textit{Gamma}(\alpha,\beta)$ random variable. Thus, we see that $B_{n, \re \lambda}$ has a $\textit{Gamma}(n,n \re \lambda)$ distribution. A $\textit{Gamma}(n,n \re \lambda)$ random variable, can be obtained as the $n$-fold convolution of $\textit{Gamma}(1,n \re \lambda)$ random variables, i.e. exponential random variables with parameter $n \re \lambda$. Probabilistically, this means that a $\textit{Gamma}(n,n \re \lambda)$ can be written as the sum of $n$ independent exponential random variables with parameter $n \re \lambda$.
An exponential random variable $\eta$ that is $\textit{Exp}(\beta)$ distributed has the property that $\frac{1}{n} \eta$ is $\textit{Exp}(n \beta)$ distributed. Therefore, we obtain that $B_{n, \re \lambda} = \frac{1}{n} \sum_{i = 1}^n X_{i, \re \lambda}$ where $\{X_{i,\beta}\}_{i \geq 1}$ are independent copies of an $\textit{Exp}(\beta)$ random variable $X_\beta$.

\smallskip

This implies that we are in a position to use a Chernoff bound to control the tail probabilities of the $B_{n, \re \lambda}$. An elementary calculation shows that for $0 < \theta < (\re \lambda)$, we have $\bE[e^{\theta X_{\re \lambda}}] = \frac{\re \lambda}{\re \lambda - \theta}$. Evaluating the infimum in Chernoff's bound yields for $c \geq (\re \lambda)^{-1}$ that
\begin{equation*}
\PR[B_{n, \re \lambda} > c] < e^{-n \left(c\re \lambda - 1 - \log c \re \lambda \right)}.
\end{equation*}

Define the non-negative function 
\begin{gather*}
\phi : [\lambda_0^{-1},\infty)\times [\lambda_0,\infty)  \rightarrow [0,\infty)  \\
(c,\alpha)  \mapsto c\alpha - 1 - \log c\alpha 
\end{gather*}
so that for $c \geq \lambda_0^{-1}$ and  $\lambda$ such that $\re \lambda \geq \lambda_0$ we have
\begin{equation} \label{eqn:chernoff_exp}
\PR[B_{n, \re \lambda} > c]  < e^{-n\phi(c,\re \lambda)}. 
\end{equation}
We use this result to find a random variable that stochastically dominates all $B_{n,\re \lambda}$ for $n \in \bN$ and $\re \lambda \geq \lambda_0$. Define the random variable $Y$ on $[\lambda_0^{-1},\infty)$ by setting $\PR[Y > c] = \exp\{-\phi(c,\lambda_0)\}$.

First note that for fixed $c \geq \lambda_0^{-1}$, the function $\alpha \mapsto \phi(c,\alpha)$ is increasing. Also note that $\phi \geq 0$. Therefore, it follows by Equation \eqref{eqn:chernoff_exp} that for $\lambda$ such that $\re \lambda \geq \lambda_0$ and $c \geq \lambda_0^{-1}$, we have
\begin{equation*}
\PR[B_{n, \re \lambda} > c]  < e^{-n\phi(c,\re \lambda)}  \leq e^{-\phi(c,\re \lambda)} \leq e^{-\phi(c,\lambda_0)} = \PR[Y > c].
\end{equation*}
For $0 \leq c \leq \lambda_0^{-1}$, $\PR[Y > c] =1$ by definition, so clearly $\PR[B_{n, \re \lambda} > c] \leq \PR[Y > c]$. Combining these two statements gives $Y \succeq B_{n, \re \lambda}$ for $n \geq 1$ and $\lambda$ such that $\re \lambda \geq \lambda_0$. This implies by Lemma \ref{lemma:appendix_stochastic_domination_equivalence} that
\begin{equation*}
p\left(\left(n \re \lambda R(n \lambda)\right)^n x\right) \leq \bE\left[q_{\lceil B_{n, \re \lambda}  \rceil}(x) \right] \leq  \bE\left[q_{\lceil Y  \rceil}(x) \right]  =: q(x)
\end{equation*}
By the countable convexity of $\cN$, $q$ is continuous and in $\cN$, which proves the second statement of the theorem.

\smallskip

The strengthening to the case where the semigroup is of type $(M,\omega)^*$ is obvious, as it is sufficient to consider just one semi-norm $q \in \cN$ for every $p \in \cN$.
\end{proof}

\section{Generation results} \label{section:generation_results}

The goal of this section is to prove a Hille-Yosida result for locally equi-continuous semigroups. First, we start with a basic generation result for the semigroup generated by a continuous linear operator.

\begin{lemma} \label{lemma:generation_by_continuous_operator}
Let $(X,\tau,\vn{\cdot})$ satisfy Condition C. Suppose we have some $\tau$ continuous and linear operator $G : X \rightarrow X$. Then $G$ generates a SCLE semigroup defined by
\begin{equation} \label{equation:definition_St}
S(t)x := \sum_{k \geq 0} \frac{t^k G^k x}{k!}.
\end{equation}
\end{lemma}

\begin{proof}
First, we show that the infinite sum in Equation \eqref{equation:definition_St} is well defined. As $\tau$ is weaker than the norm-topology, it is sufficient  to prove that the sum exists as a norm limit. By Condition C and Lemma \ref{lemma:equivalences_boundedness}, $(X,\vn{\cdot})$ is a Banach space. Therefore, we need to show for some fixed $t \geq 0$ and $x \in X$ that the sequence $y_n = \sum_{k = 0}^n \frac{t^k G^k x}{k!}$ is Cauchy for $\vn{\cdot}$. Note that as $G$ is $\tau$ continuous, it is also norm continuous. Suppose that $n \geq m$, then we have
\begin{align*}
\vn{y_n - y_m} & \leq \sum_{m < k \leq n} \frac{t^k}{k!} \vn{G}^k \vn{x} \\
& \leq \vn{x} \sum_{k > m} \frac{t^k}{k!} \vn{G}^k.
\end{align*}
which can be made arbitrarily small by choosing $m$ large.

\smallskip

We proceed with showing that the $\tau$ continuous operators $S_n(t) : X \rightarrow X$ defined for $x \in X$ by $S_n(t)x := \sum_{k = 0}^n \frac{t^k G^k x}{k!}$ are equi-continuous. As in the proof of Lemma \ref{lemma:equivalence_norm_and_N_bounds}, the fact that $G$ is $\tau$ continuous implies that for every $p \in \cN$, there exists $q \in \cN$ such that $p(Gx) \leq \vn{G} q(x)$ for all $x \in X$. Use this method to construct for a given $p \in \cN$ an increasing sequence of semi-norms $q_n \in \cN$, $q_0 := p$, such that $q_n(Gx) \leq \vn{G} q_{n+1}(x)$ for every $n \geq 0$ and $x \in X$. As a consequence, we obtain
\begin{align*}
p(S_n(t)x) = p \left(\sum_{k = 0}^n \frac{t^k G^k x}{k!} \right) & \leq \sum_{k = 0}^n \frac{t^k}{k!}  p(G^k x)  \\
& \leq \sum_{k \geq 0} \frac{t^k}{k!} p(G^k x) \\
& \leq e^{t \vn{G}}\sum_{k \geq 0} \frac{(\vn{G} t)^k}{k!} e^{-t \vn{G}} q_k(x)  \\
& \leq  e^{t \vn{G}} q_t(x), 
\end{align*}
where 
\begin{equation*}
q_t(x) := \sum_{k \geq 0} \frac{(\vn{G} t)^k}{k!} e^{-t \vn{G}} q_k(x)
\end{equation*}
is a continuous semi-norm in $\cN$ by Condition C (d). The semi-norm $q_t$ is independent of $n$ which implies that $\{S_n(t)\}_{n \geq 1}$ is $\tau$ equi-continuous. It follows that $S(t)$ is $\tau$ continuous: pick a net $x_\alpha$ in $X$ that converges to $x \in X$ with respect to $\tau$. Let $p \in \cN$, then
\begin{align*}
& p\left(S(t)x_\alpha - S(t)x\right) \\
& \quad \leq p\left(S(t)x_\alpha - S_n(t)x_\alpha\right) + p\left(S_n(t)x_\alpha - S_n(t)x\right) + p\left(S_n(t)x - S(t)x_\alpha\right) \\
& \quad \leq p\left(S(t)x_\alpha - S_n(t)x_\alpha\right) + q_t\left(x_\alpha - x\right) + p\left(S_n(t)x - S(t)x_\alpha\right).
\end{align*}
By first choosing $\alpha$, and then $n$ large enough, we see $p\left(S(t)x_\alpha - S(t)x\right) \rightarrow 0$.

By stochastic domination of Poisson random variables, Lemmas \ref{lemma:appendix_stochastic_domination_equivalence} and \ref{lemma:appendix_domination_poisson}, it follows that for $t \leq T$, we have that
\begin{equation*}
\sup_{t \leq T} e^{-t \vn{G}} p(S(t)x) \leq \sup_{t \leq T} q_t(x) = q_T(x).
\end{equation*}
To prove strong continuity, it suffices to check that
$\lim_{t \downarrow 0} S(t)x = x$ for every $x \in X$ by Lemma \ref{lemma:equivalences_strong_continuity_2}.
To that end again consider $p \in \cN$, we see
\begin{equation*}
p(S(t)x - x)  = p \left(\sum_{k \geq 0} \frac{t^k G^k x}{k!} - x \right) \leq \sum_{k \geq 0} \frac{t^k}{k!} p(G^k x - x).
\end{equation*}
Note that the first order term vanishes. Therefore, the Dominated convergence theorem implies that the limit is $0$ as $t \downarrow 0$.
\end{proof}

In the proof of the Hille-Yosida theorem on Banach spaces, the semigroup is constructed as the limit of semigroups generated by the Yosida approximants. In the locally convex context, we need to take special care of equi-continuity of the approximating semigroups. 

Suppose we would like to generate a locally equi-continuous semigroup $e^{tA}$ for some operator operator $(A,\cD(A))$. The next lemma will yield joint local equi-continuity of the semigroups generated by the Yosida approximants by taking $H_n = nR(n,A)$.

\begin{lemma} \label{lemma:joint_local_equi_continuity}
Let $(X,\tau)$ satisfy Condition C. Let $\{H_n\}_{n \geq 1}$ be a family of operators in $\cL(X,\tau)$ such that for every $p \in \cN$ there is $q \in \cN$ such that
\begin{equation} \label{eqn:joint_local_equi_cont_condition}
\sup_{n \geq 1} \sup_{k \leq n} p(H_n^k x) \leq q(x)  
\end{equation}
for all $x \in X$. Then, the semigroups $e^{t(n H_n - n)}$ are jointly locally equi-continuous.
\end{lemma} 

\begin{proof}
By Lemma \ref{lemma:generation_by_continuous_operator}, we can define the semigroups $S_n(t) := e^{t(n H_n - n)}$. We see that
\begin{equation*}
S_n(t)x := \sum_{k \geq 0} \frac{(tn)^k H_n^k x}{k!}e^{-tn},
\end{equation*}
which intuitively corresponds to taking the expectation of $k \mapsto H_n^k x$ under the law of a Poisson random variable with parameter $n$. We exploit this point of view to show equi-continuity of the family $\{S_n(t)\}_{t \leq T, n \geq 1}$ for some arbitrary fixed time $T \geq 0$.

Let $\{Z_\mu\}_{\mu \geq 0}$ be a family of independent random variables, where $Z_\mu$ has a $\textit{Poisson}(\mu)$ distribution. For $t \geq 0$ and $n \geq 1$ let $B_{n,t} := \lceil\frac{Z_{nt}}{n} \rceil$. The random variable $B_{n,t}$ is obtained from $Z_{nt}$ as follows: $0$ is mapped to $0$, and the values $\{nl+k\}_{k=1}^{n}$ are mapped to $l + 1$. 
Fix a semi-norm $p \in \cN$ and use Equation \eqref{eqn:joint_local_equi_cont_condition} to construct an increasing sequence of semi-norms in $\cN$: $q_0 = p, q_1, \dots$ such that every pair $q_l, q_{l+1}$ satisfies the relation in \eqref{eqn:joint_local_equi_cont_condition}. As a consequence, we obtain
\begin{equation} \label{eqn:calculationPoisson}
\begin{aligned}
& p(S_n(t)) \\
& \leq p\left(\sum_{k\geq 0}  \frac{(tn)^k H_n^k x}{k!} e^{-tn}\right) \\
& \leq p(x)e^{-tn} + \sum_{l \geq 0} \sum_{k = 1}^n \frac{(tn)^{nl + k}}{(nl + k)!} e^{-tn} p\left(H_n^{nl + k} x \right) \\
& \leq q_0(x)e^{-tn} +  \sum_{l \geq 0} \sum_{k = 1}^n \frac{(tn)^{nl + k}}{(nl + k)!} e^{-tn} q_{l+1}(x) \\
& = \PR[B_{n,t} = 0]q_0(x) + \sum_{l \geq 0} \PR[B_{n,t} = l+1] q_{l+1}(x) \\
& = \bE\left[q_{B_{n,t}}(x) \right].
\end{aligned}
\end{equation}

We see that, as in the proof of Theorem \ref{theorem:EN.II.1.10}, we are done if we can find a random variable $Y$ that stochastically dominates all $B_{n,t}$ for $n \geq 1$ and $t \leq T$. 

We calculate the tail probabilities of $B_{n,t}$ in the case that $t > 0$. If $t = 0$, all tail probabilities are $0$. By the definition of $B_{n,t}$,
\begin{equation*}
\PR[B_{n,t} > k]  = \PR\left[Z_{nt} > nk\right].
\end{equation*}
As $Z_{nt}$ is $\textit{Poisson}(nt)$ distributed, we can write it as $Z_{nt} = \sum_{i=1}^n X_i$ where $\{X_i\}_{i \geq 0}$ are independent and $\textit{Poisson}(t)$ distributed. This implies that we can apply Chernoff's bound to $\frac{1}{n}Z_{nt}$, see Proposition \ref{proposition:Chernoff}. First of all, for all $\theta \in \bR$, we have $\bE\left[e^{\theta X} \right] =  \exp\{t(e^\theta - 1)\}$. Evaluating the infimum in Chernoff's bound for $k \geq \lceil T \rceil$, $T \geq t$ yields
\begin{equation*}
\PR[B_{n,t} > k] = \PR\left[\frac{1}{n}Z_{nt} > k\right]  < e^{-n\left(k \log \frac{k}{t} - k + t \right)}. 
\end{equation*}

Define the function 
\begin{gather*}
\phi : [\lceil T \rceil,\infty)\times (0,T]  \rightarrow [0,\infty) \\
(a,b)  \mapsto a \log \frac{a}{b} - a + b,
\end{gather*}
so that for $k \geq \lceil T \rceil$, $T \geq t$, we have $\PR[B_{n,t} > k]  < e^{-n\phi(k,t)}$.

We define a new random variable $Y$ taking values in $\{n \in \bN \, | \, n \geq \lceil T \rceil \}$ by putting $\PR[Y = \lceil T \rceil] = 1 - e^{-\phi(\lceil T \rceil, T)}$, and for $k \geq \lceil T \rceil$: $\PR[Y > k] = e^{-\phi(k,T)}$, or stated equivalently $\PR[Y = k + 1] = e^{-\phi(k,T)} -  e^{-\phi(k+1,T)}$.

For $k < \lceil T \rceil$, we have by definition that $\PR[Y > k] \geq \PR[B_{n,t} > k]$ as the probability on the left is $1$. For $k \geq \lceil T \rceil$, an elementary computation shows that for fixed $k$ and $t \leq T$ the function $\phi(k,t)$ is decreasing in $t$. This implies that
\begin{equation*}
\PR[B_{n,t} > k]  \leq e^{-n\phi(k,t)} \leq e^{-\phi(k,t)} \leq e^{-\phi(k,T)}  = \PR[Y > k].
\end{equation*}
In other words, we see $Y \succeq B_{n,t}$ for all $n \geq 1$ and $0 < t \leq T$. For the remaining cases, where $t = 0$, the result is clear as $B_{n,t} = 0$ with probability $1$. By Lemma \ref{lemma:appendix_stochastic_domination_equivalence} and the bound in \eqref{eqn:calculationPoisson}, we obtain that
\begin{equation*}
p(S_n(t)) \leq \bE\left[q_{B_{n,t}}(x) \right]  \leq \bE\left[q_{Y}(x) \right]  =: q(x).
\end{equation*}
The semi-norm $q(x)$ is in $\cN$ by the countable convexity of $\cN$.
We conclude that the family $\{S_n(t)\}_{t \leq T, n \geq 1}$ is equi-continuous.
\end{proof}

\begin{lemma} \label{lemma:resolvent_convergence}
Let $(X,\tau,\vn{\cdot})$ satisfy Condition C. Let $(A,\cD(A))$ be a closed, densely defined operator such that there exists an $\omega \in \bR$ such that $(\omega,\infty) \subseteq \rho(A)$ and such that for every $\lambda_0 > \omega$ and semi-norm $p \in \cN$, there is a continuous semi-norm $q$ such that $\sup_{\lambda \geq \lambda_0} p((\lambda- \omega) R(\lambda)x) \leq q(x)$  for every $x \in X$. As $\lambda \rightarrow \infty$, we have
\begin{enumerate}[(a)]
\item $\lambda R(\lambda)x \rightarrow x$ for every $x \in X$,
\item $\lambda A R(\lambda)x = \lambda R(\lambda) A x \rightarrow Ax$ for every $x \in \cD(A)$.
\end{enumerate}
\end{lemma}

The lemma can be proven as in the Banach space case\cite[Lemma II.3.4]{EN00}. We have now developed enough machinery to prove a Hille-Yosida type theorem which resembles the equivalence between (a) and (b) of Theorem 16 in \cite{Ku03}.

\begin{theorem} \label{theorem:Hille_Yosida}
Let $(X,\tau,\vn{\cdot})$ satisfy Condition C. For a linear operator $(A,\cD(A))$ on $(X,\tau)$, the following are equivalent.
\begin{enumerate}[(a)]
\item $(A,\cD(A))$ generates a SCLE semigroup of type $(M,\omega)$.
\item $(A,\cD(A))$ is closed, densely defined and there exists $\omega \in \bR$ and $M \geq 1$ such that for every $\lambda > \omega$ one has $\lambda \in \rho(A)$ and for every semi-norm $p \in \cN$ and $\lambda_0 > \omega$ there exists a semi-norm $q \in \cN$ such that for all $x \in X$ one has 
\begin{equation}\label{eqn:Hille_Yosida_estimate_not_complex}
\sup_{n \geq 1} \sup_{\lambda \geq \lambda_0} p\left( \left(n (\lambda- \omega) R(n \lambda)\right)^n x \right) \leq M q(x).
\end{equation}
\item $(A,\cD(A))$ is closed, densely defined and there exists $\omega \in \bR$ and $M \geq 1$ such that for every $\lambda \in \bC$ satisfying $\re \lambda > \omega$, one has $\lambda \in \rho(A)$ and for every semi-norm $p \in \cN$ and $\lambda_0 > \omega$ there exists a semi-norm $q \in \cN$ such that for all $x \in X$ and $n \in \bN$ 
\begin{equation*} 
\sup_{n \geq 1}\sup_{\re \lambda \geq \lambda_0} p\left( \left(n(\re \lambda - \omega) R(n \lambda)\right)^n x \right) \leq M q(x).
\end{equation*}
\end{enumerate}
\end{theorem}

By a simplification of the arguments, or arguing as in Section IX.7 in Yosida\cite{Yo78}, we can also give a necessary and sufficient condition for the generation of a quasi equi-continuous semigroup of type $(M,\omega)^*$, which corresponds with the result obtained in Theorem \ref{theorem:EN.II.1.10}. Theorem 3.5 in Kunze \cite{Ku09} states a similar result.

\smallskip

Suppose we have a semigroup of type $(M,\omega)$ and let $\omega' > \omega$. Equation \eqref{eqn:Hille_Yosida_estimate_not_complex} yields
\begin{equation} \label{eqn:bound_resolvent_quasi_equicontinuous}
\sup_{n \geq 1} \sup_{\lambda > \omega'} p\left( \left(n (\lambda - \omega') R(n \lambda)\right)^n x \right) \leq M q(x)
\end{equation}
which implies that the semigroup is of type $(M,\omega')^*$. We state this as a corollary.

\begin{corollary}
Suppose that $(X,\tau,\vn{\cdot})$ satisfies Condition C. If a semigroup is of type $(M,\omega)$, then it is of type $(M,\omega')^*$ for all $\omega' > \omega$. 
\end{corollary}

As Equation \eqref{eqn:Hille_Yosida_estimate_not_complex} implies \eqref{eqn:bound_resolvent_quasi_equicontinuous}, it is sufficient, for the construction of a semigroup, to use the weaker result as in Kunze\cite{Ku09}. However, one obtains that the semigroup is of type $(M,\omega')$ for $\omega' > \omega$, which does not give any control if the semigroup is rescaled by $e^{-\omega t}$. A semigroup that is of type $(M',\omega')$ for all $\omega' > \omega$ is not necessarily of type $(M,\omega)$ for any $M \geq M'$ as is shown in Example I.5.7(ii) in \cite{EN00}.

The proof of the Hille-Yosida theorem stated here, however, gives explicit control on the semigroup rescaled by $e^{-\omega t}$ via the construction in Lemma \ref{lemma:joint_local_equi_continuity} and gives a result as strong as the equivalence of (a) and (b) of Theorem 16 in \cite{Ku03}. 

\begin{proof}[Proof of Theorem \ref{theorem:Hille_Yosida}]
(a) to (c) is the content of Theorem \ref{theorem:EN.II.1.10} and (c) to (b) is clear. So we need to prove (b) to (a).

\smallskip

First note that we can always assume that $\omega = 0$ by a suitable rescaling. We start by proving the result for $\omega = 0$ and $M = 1$. We follow the lines of the proof of the Hille-Yosida theorem for Banach spaces in Engel and Nagel \cite[Theorem II.3.5]{EN00}. 

Define for every $n \in \bN\setminus\{0\}$ the Yosida approximants
\begin{equation*}
A_n := n A R(n) = n^2  R(n) - n \bONE.
\end{equation*}
These operators commute and for every $n$ $A_n$ satisfies the condition in Lemma \ref{lemma:generation_by_continuous_operator}. Furthermore, we can apply Lemma \ref{lemma:joint_local_equi_continuity} to the operators $H_n = n R(n)$. Note that Equation \eqref{eqn:joint_local_equi_cont_condition} is satisfied as a consequence of Equation \eqref{eqn:Hille_Yosida_estimate_not_complex}, as the latter implies
\begin{equation*}
\sup_k \sup_{\lambda \in \{\frac{n}{k} \, | \, n \geq k \}} p\left( (\lambda k R(\lambda k))^k x\right) \leq q(x)
\end{equation*}
for all $x$, which in turn can be rewritten to
\begin{equation*}
\sup_n \sup_{k \leq n} p\left((nR(n))^k x\right) \leq q(x)
\end{equation*}
for all $x \in X$. Hence, we obtain that the operators $A_n$ generate jointly locally equi-continuous strongly continuous commuting semigroups $t \mapsto T_n(t)$ of type $(1,0)$. We show that there exists a limiting semigroup.

Let $x \in \cD(A)$ and $t \geq 0$, the fundamental theorem of calculus applied to $s \mapsto T_m(t-s)T_n(s) x$ for $s \leq t$, yields
\begin{align*}
T_n(t)x - T_m(t)x & = \int_0^t T_m(t-s)\left(A_n - A_m\right)T_n(s) x \dd s \\
& = \int_0^t T_m(t-s)T_n(s) \left(A_n x - A_m x\right) \dd s.
\end{align*}
By Lemma \ref{lemma:joint_local_equi_continuity}, we obtain that for every semi-norm $p \in \cN$ there exists $q \in \cN$ such that
\begin{equation} \label{eqn:uniform_bound_on_difference_Yosida_approximants}
p(T_n(t)x - T_m(t)x) \leq t q(A_n x - A_m x).
\end{equation}
Hence, for $x \in \cD(A)$ the sequence $n \mapsto T_n(s)x$ is $\tau$-Cauchy uniformly for $s \leq t$ by Lemma \ref{lemma:resolvent_convergence} (b). The joint local equi-continuity of $\{T_n(t)\}_{t \geq 0, n \geq 1}$ implies that this property extends to all $x \in X$. 

Define the point-wise limit of this sequence by $T(s)x := \lim_n T_n(s)x$. This directly yields that the family $\{T(s)\}_{s \leq t}$ is equi-continuous, because it is contained in the closure of an equi-continuous set of operators, Proposition 32.4 in Treves\cite{Tr67}. 
Consequently, this shows that $\{T(t)\}_{t \geq 0}$ is a locally equi-continuous set of operators of type $(1,0)$. 

\smallskip

The family of operators $\{T(t)\}_{t \geq 0}$ is a semigroup, because it is the point-wise limit of the semigroups $\{T_n(t)\}_{t \geq 0}$. We show that it is strongly continuous by using Lemma \ref{lemma:equivalences_strong_continuity_2}. Let $p \in \cN$ and $x \in \cD(A)$, then for every $n$: 
\begin{equation*}
p(T(t)x - x) \leq p(T(t)x - T_n(t)x) + p(T_n(t) x - x).
\end{equation*}
As $p(T(t)x - T_n(t)x) \rightarrow 0$, uniformly for $t \leq 1$, we can first choose $n$ large to make the first term on the right hand side small, and then $t$ small, to make the second term on the right hand side small.

\smallskip

We still need to prove that the semigroup $\{T(t)\}_{t \geq 0}$ has generator $(A,\cD(A))$. Denote with $(B,\cD(B))$ the generator of $\{T(t)\}_{t \geq 0}$. For $x \in \cD(A)$, we have for a continuous semi-norm $p$ that
\begin{align*}
& p\left(\frac{T(t)x - x}{t} - Ax\right) \\
& \leq p\left(\frac{T(t)x - T_n(t)x}{t} \right) + p\left(\frac{T_n(t)x - x}{t} - A_n x\right) + p(A_n x - A x),
\end{align*}
for some continuous semi-norm $q$. By repeating the argument that led to \eqref{eqn:uniform_bound_on_difference_Yosida_approximants}, we can rewrite the first term on the second line to obtain
\begin{align*}
& p\left(\frac{T(t)x - x}{t} - Ax\right) \\
& \leq q\left(Ax - A_n x\right) + p\left(\frac{T_n(t)x - x}{t} - A_n x\right) + p(A_n x - A x).
\end{align*}
By first choosing $n$ large and then $t$ small, we see that $x \in \cD(B)$ and $Bx = Ax$. In other words, $(B,\cD(B))$ extends $(A,\cD(A))$.

For $\lambda >0$, we know that $\lambda \in \rho(A)$, so $\lambda - A : \cD(A) \rightarrow X$ is bijective. As $B$ generates a semigroup of type $(1,0)$, we also have that $\lambda - B : \cD(B) \rightarrow X$ is bijective. But $B$ extends $A$, which implies that $(A,\cD(A)) = (B,\cD(B))$.

\smallskip

We extend the result for general $M \geq 1$. The strategy is to define a norm on $X$ that is equivalent to $\vn{\cdot}$ for which the semigroup that we want to construct is $(1,0)$ bounded.
Equations \eqref{eqn:norm_given_as_sup_over_N} and \eqref{eqn:Hille_Yosida_estimate_not_complex} imply that $\vn{\mu^n R(\mu)^n} \leq M$. Define
\begin{equation*}
\vn{x}_\mu := \sup_{n \geq 0} \vn{\mu^n R(\mu)^n x}
\end{equation*}
and then define $\tn{x} := \sup_{\mu > 0} \vn{x}_\mu$. This norm has the property that $\vn{x} \leq \tn{x} \leq M \vn{x}$ and $\tn{\lambda R(\lambda)} \leq 1$ for every $\lambda > 0$, see the proof of Theorem II.3.8 in \cite{EN00}. Use this norm to define a new set of continuous semi-norms as in Definition \ref{definition:def_of_N} by
\begin{equation*}
\cN^* := \{p \, | \, p \textit{ is a $\tau$ continuous semi-norm such that } p(\cdot) \leq \tn{\cdot}\}.
\end{equation*}

As a consequence of $\tn{\lambda R(\lambda)} \leq 1$ and the $\tau$ continuity of $\lambda R(\lambda)$, we obtain that for every $p \in \cN^*$ there exists $q \in \cN^*$ such that $p(\lambda R(\lambda)x) \leq q(x)$ for every $x \in X$. Likewise, we obtain for every $\lambda_0 > 0$ that for every $p \in \cN^*$ there exists $q \in \cN^*$ such that
\begin{equation*}
\sup_{\lambda \geq \lambda_0} \sup_{n \geq 1} p\left((n\lambda R(n\lambda))^n x \right) \leq q(x).
\end{equation*}

This means that we can use the first part of the proof to construct a SCLE semigroup $\{T(t)\}_{t \geq 0}$ that has bound $(1,0)$ with respect to $\cN^*$. 

Let $T \geq 0$. Pick a semi-norm $p \in \cN$. It follows that $p \in \cN^*$, so there exists a $q \in \cN^*$ such that $\sup_{t \leq T} p(T(t)x) \leq q(x)$ for all $x \in X$.

Because $\tn{\cdot} \leq M \vn{\cdot}$, it follows that $\cN^*$ is a subset of $M \cN$ which implies that $\hat{q} := \frac{1}{M} q \in \cN$. We obtain $\sup_{t \leq T} p(T(t)x) \leq M \hat{q}(x)$ for all $x \in X$.

In other words, $A$ generates a SCLE semigroup $\{T(t)\}_{t \geq 0}$ of type $(M,0)$.
\end{proof}

\section{Relating bi-continuous semigroups to SCLE semigroups} \label{section:comparison_bicontinuous}

Bi-continuous semigroups were introduced by K\"{u}hnemund\cite{Ku03} to study semigroups on Banach spaces that are strongly continuous with respect to a weaker locally convex topology $\tau$ and where $\tau$ has good sequential properties on norm bounded sets.
We will consider the \textit{mixed topology} $\gamma := \gamma(\vn{\cdot},\tau)$, introduced by Wiweger \cite{Wi61}, also see \cite{Co87}, which is the strongest locally convex topology that coincides with $\tau$ on norm bounded sets. We will show that if $\gamma$ satisfies $\gamma^+ = \gamma$, then bi-continuity of a semigroup for $\tau$ is equivalent to being SCLE for $\gamma$.

We start with the assumptions underlying bi-continuous semigroups.

\begin{condition} \label{condition:topology_for_bi_cont}
Let $(X,\vn{\cdot})$ be a Banach space with continuous dual $X_n'$ and dual unit ball $B_n$. Let $\tau$ be another, coarser, locally convex topology on $X$, with continuous dual $X_\tau'$ and dual unit ball $B_\tau = B_n \cap X_\tau'$ that has the following two properties.
\begin{enumerate}[(a)]
\item The space $(X,\tau)$ is sequentially complete on norm bounded sets.
\item $X_\tau'$ is norming for $(X,\vn{\cdot})$, i.e. $\vn{x} = \sup_{x' \in B_\tau} |\ip{x}{x'}|$.
\end{enumerate}
\end{condition}

An operator family $\{T(t)\}_{t \geq 0}$ of norm continuous operators on $X$ is called \textit{locally bi-continuous} if for any $t_0 \geq 0$ and for any norm bounded sequence $\{x_n\}_{n \geq 0}$ that converges to $x$ in $X$ with respect to $\tau$, we have
\begin{equation*}
\tau - \lim_{n \rightarrow \infty} T(t)(x_n - x) = 0
\end{equation*} 
uniformly for $t \leq t_0$. K\"{u}hnemund\cite{Ku03} then introduces bi-continuous semigroups.
 
\begin{definition}
A semigroup $\{T(t)\}_{t \geq 0}$ of norm continuous operators on $X$ is called a \textit{bi-continuous semigroup} of type $(M,\omega)$ if it satisfies the following properties.
\begin{enumerate}[(a)]
\item $\{T(t)\}_{t \geq 0}$ is $\tau$ strongly continuous.
\item $\{T(t)\}_{t \geq 0}$ is locally bi-continuous as an operator family.
\item The semigroup is exponentially bounded: $\vn{T(t)} \leq M e^{\omega t}$ for all $t \geq 0$.
\end{enumerate}
\end{definition}

We will compare bi-continuous semigroups for $\tau$ with SCLE semigroups for mixed topology $\gamma := \gamma(\vn{\cdot},\tau)$.

\begin{proposition}
Let $(X,\tau,\vn{\cdot})$ satisfy Condition \ref{condition:topology_for_bi_cont}. Then, $\gamma$ is sequentially complete and has the same bounded sets as the norm topology.
If $\gamma$ satisfies $\gamma^+ = \gamma$, then $(X,\gamma,\vn{\cdot})$ satisfies Condition C.
\end{proposition}

\begin{proof}
By Condition \ref{condition:topology_for_bi_cont}, $\tau$ and $\vn{\cdot}$ satisfy properties (n), (o) and (d) in \cite{Wi61}. Thus, it follows by the Corollary of 2.4.1 in \cite{Wi61} that the $\gamma$ bounded sets equal the norm bounded sets.

By 2.2.1 in \cite {Wi61}, $\gamma$ coincides with $\tau$ on norm bounded sets, which implies that $\gamma$ is sequentially complete.

The final statement follows from Proposition \ref{prop:what_implies_type_C}.
\end{proof}
 
The definition of bi-continuous semigroups is given using the convergence of sequences. Therefore, we expect a connection to SCLE semigroups if $\gamma^+ = \gamma$.

\begin{theorem} \label{theorem:equality_bicontinuous_SCLE}
Let $(X,\tau,\vn{\cdot})$ satisfy Condition \ref{condition:topology_for_bi_cont} and let $\gamma$ be such that $\gamma^+ = \gamma$. $\{T(t)\}_{t \geq 0}$ is bi-continuous for $\tau$ if and only if it is SCLE for $\gamma$.
\end{theorem}

This theorem is an extension of Theorem 3.4 in \cite{Fa11}, see also Section \ref{subsection:classical_topology}.

\begin{proof}
Let $\{T(t)\}_{t \geq 0}$ be bi-continuous for $\tau$. Fix $t_0 > 0$. As $\sup_{t \leq t_0} \vn{T(t)} < \infty$, it follows from 2.2.1 in \cite {Wi61} and the $\tau$ strong continuity of $\{T(t)\}_{t \geq 0}$ that the semigroup is also $\gamma$ strongly continuous.

\smallskip

As a $\gamma$ converging sequence is norm bounded, it converges for $\tau$. Thus $\{T(t)\}_{t \leq t_0}$ is sequentially equi-continuous for $\gamma$ by the local bi-continuity of $\{T(t)\}_{t \geq 0}$. It follows that for a $\gamma$ continuous semi-norm $p$ there exists a sequentially continuous semi-norm $q$ such that
\begin{equation*}
\sup_{t \leq t_0} \p{T(t)x} \leq q(x) 
\end{equation*}
for all $x \in X$. However, using that $\gamma^+ = \gamma$ and Theorem 7.4 in \cite{Wi81}, $q$ is $\gamma$ continuous. In other words, $\{T(t)\}_{t \geq 0}$ is locally equi-continuous.

\smallskip

Now let Let $\{T(t)\}_{t \geq 0}$ be SCLE for $\gamma$. The semigroup is exponentially bounded by Corollary \ref{corollary:exponential_continuity}. Thus, 2.2.1 in \cite {Wi61} implies that $t \mapsto T(t)x$ is $\tau$ continuous for every $x \in X$ and that $\{T(t)\}_{t \geq 0}$ is $\tau$ locally bi-continuous.
\end{proof}

\section{The strict topology} \label{section:application_markov}

We give two examples where a strict topology can be defined. In both cases, this topology is strongly Mackey and satisfies Condition C.

For the first example, let $E$ be a Polish space. We will define the \textit{strict topology} $\beta$ on $C_b(E)$ which is a particularly nice topology as the continuous dual of $(C_b(E),\beta)$ is the space of Radon measures on $E$ of finite total variation. Therefore, this topology is useful for, for example, the study of transition semigroups of Markov processes.

For the second example, we take a Hilbert space $\fH$ and consider the \textit{strict topology} $\beta$ on $\cB(\fH)$, the space of bounded operators on $\fH$. The dual of $(\cB(\fH),\beta)$ is the space of normal linear functionals, which are at the basis of non-commutative measure theory\cite{Ta79,KR86,BR79}. As a consequence, the space $(\cB(\fH),\beta)$ is suitable for the study of quantum dynamical semigroups.

\subsection{Definition and basic properties of the strict topology on \texorpdfstring{$C_b(E)$}{Cb(E)}} \label{subsection:classical_topology}

For every compact set $K \subseteq E$, define the semi-norm $p_K(f) := \sup_{x \in K} |f(x)|$.

The \textit{compact open} topology $\kappa$ on $C_b(E)$ is generated by the semi-norms $\{ p_K \, | \, K \text{ compact} \}$. Now define semi-norms in the following way. Pick a non-negative sequence $a_n$ in $\bR$ such that $a_n \rightarrow 0$. Also pick an arbitrary sequence of compact sets $K_n \subseteq E$. Define 
\begin{equation} \label{eqn:def_seminorm_strict}
p_{(K_n),(a_n)}(f) := \sup_n a_n p_{K_n}(f).
\end{equation}
The \textit{strict} topology $\beta = \gamma(\vn{\cdot},\kappa)$ defined on $C_b(E)$ is generated by the semi-norms
\begin{equation*}
\left\{p_{(K_n),(a_n)} \, \middle| \, K_n \text{ compact}, a_n \geq 0, a_n \rightarrow 0 \right\},
\end{equation*}
see Theorem 3.1.1 in Wiweger \cite{Wi61} and Theorem 2.4 in Sentilles \cite{Se72}. Note that in the latter paper, the topology introduced here is called the substrict topology. However, Sentilles shows in Theorem 9.1 that the strict and the substrict topology coincide when the underlying space $E$ is Polish.

Note that if additionally $(E,d)$ is locally compact, then the topology can also be given by the collection of semi-norms
\begin{equation*} 
p_g(f) := \vn{fg}
\end{equation*}
where $g$ ranges over $C_0(E)$.

\smallskip

Obviously, $C_b(E)$ can also equipped with the sup norm topology. In this situation, the set $\cN$ contains all semi-norms of the type given in Equation \eqref{eqn:def_seminorm_strict} such that $\sup_n a_n \leq 1$. 

Sentilles \cite{Se72} studied the strict topology and gives, amongst many others, the following results.

\begin{theorem} \label{theorem:propertiesCbstrict}
The space $(C_b(E),\beta)$ is complete, Mackey, satisfies $\beta^+ = \beta$ and has the same bounded sets as the norm topology.

Also, the dual of $(C_b(E), \beta)$ is the space of Radon measures of finite total variation norm.
\end{theorem}

\begin{proof}
These statements follow from Theorems 4.7, 8.1 and 9.1 in Sentilles\cite{Se72}.
\end{proof}

The next result follows directly from Propositions \ref{proposition:what_implies_compactequicont} and \ref{prop:what_implies_type_C}.

\begin{corollary}
The locally convex space $(C_b(E),\beta)$ together with the sup norm is strong Mackey and satisfies Condition C.
\end{corollary}

\subsection{Definition and basic properties of the strict topology on \texorpdfstring{$\cB(H)$}{B(H)}} \label{subsection:quantum_topology}

Let $\fH$ be a Hilbert space and let $(\cB(\fH),\vn{\cdot})$ be the Banach space of bounded linear operators on $\fH$. Furthermore, let $\cK(\fH)$ and $\cT(\fH)$ be the subspace of compact and trace class operators on $\fH$. Note that $\cB(\fH) = \cT(\fH)' = \cK(\fH)''$ as Banach spaces by Theorems II.1.6 and II.1.8 in \cite{Ta79}. 

We define four additional topologies on $\cB(\fH)$.
\begin{enumerate}[(a)]
\item The \textit{strong* (operator) topology} generated by the semi-norms $\{p_\xi \, | \, \xi \in \fH\}$, where $p_\xi(A) := \sqrt{\vn{A\xi}^2 + \vn{A^*\xi}^2}$.
\item The \textit{ultraweak (operator) topology} generated by the family of semi-norms $\{p_{T} \, | \, T \in \cT(\fH)\}$, where $p_T(A) := |\Tr(AT)|$.
\item The \textit{ultrastrong* (operator) topology} generated by the family of semi-norms $\{p_{T} \, | \, T \in \cT(\fH), T \geq 0\}$, where $p_T(A) := \sqrt{\Tr(T A^* A)}$.
\item The \textit{strict topology} $\beta$ defined by the set of semi-norms $p_B(A) := \vn{AB}$ and $q_B(A) := \vn{BA}$ for compact operators $B \in \cK(\fH)$.
\end{enumerate}

The ultraweak topology is the weak topology of the dual pair $(\cB(\fH),\cT(\fH))$ and also the ultrastrong* topology is a topology of this pair, see for example Lemma II.2.4\cite{Ta79}. The strict topology is the Mackey topology of this dual pair by Theorem 3.9 in \cite{Bu68} and Corollary 2.8 in \cite{Ta70}.

\smallskip

The linear functionals on $\cB(\fH)$ that are continuous with respect to any topology of the dual pair $(\cB(\fH),\cT(\fH))$ are called \textit{normal}, to distinguish them from the larger class of linear functionals on $\cB(\fH)$ that are continuous for the norm, see also the reference that were mentioned before\cite{Ta79,KR86,BR79}. The distinction between the two classes of functionals is analogous to the difference between Radon measures on $C_b(E)$, $E$ non-compact and Polish, and the linear functionals on $C_b(E)$ that are norm continuous.

\begin{proposition}
The space $(\cB(\fH),\beta)$ is complete, strong Mackey, the bounded sets equal the operator norm bounded sets, and $(\cT(\fH),\sigma(\cT(\fH),\cB(\fH))$ is sequentially complete. 
\end{proposition}

\begin{proof}
Proposition 3.6 in Busby \cite{Bu68} gives completeness. The principle of uniform boundedness gives equality of the bounded sets. To show that $\beta$ is strong Mackey, we need to verify that the absolutely convex hull of a $\sigma(\cT(\fH),\cB(\fH))$ compact set is also compact. This follows directly from Krein's theorem, 24.5.(4) in \cite{Ko69} as the Mackey topology $\mu(\cT(\fH),\cB(\fH))$ is the Banach topology generated by the Trace norm. The final statement follows from Corollary III.5.2 in \cite{Ta79}.
\end{proof}

\begin{corollary}
The space $(\cB(\fH),\beta)$ together with the operator norm is strong Mackey and satisfies Condition C.
\end{corollary}

If $\fH$ is separable, we additionally have the following result.

\begin{proposition}
If $\fH$ is separable, then $(\cB(\fH),\beta)$ is separable and $\beta^+ = \beta$
\end{proposition}

\begin{proof}
Suppose $\fH$ is separable. Then $\cK(\fH)$ is norm separable by Lemma 1 in Goldberg \cite{Go59} which implies that it is separable for $\beta$. By Proposition 3.5 in Busby \cite{Bu68}, $\cK(\fH)$ is $\beta$ dense in $\cB(\fH)$, which implies the first statement.

By Theorem II.2.6 and Proposition II.2.7 in \cite{Ta79} $(\cB(\fH),ultrastrong^*)$ is Mazur. Consider the topology $(ultrastrong^*)^+$. By Theorem 7.5 in Wilansky \cite{Wi81}, $(ultrastrong^*)^+$ is a topology of the dual pair $(\cB(\fH),\cT(\fH))$, hence must be coarser than the strict topology. By Theorem III.5.7 in \cite{Ta79} the strict topology coincides on bounded sets with the $ultrastrong^*$ topology. Hence, both have the same convergent sequences, which implies that $(ultrastrong^*)^+$ is finer than the strict topology. Therefore, they coincide. This also implies that $\beta^+ = \beta$.
\end{proof}

Let $\{P_t\}_{t \geq 0}$ be a strongly continuous semigroup on $\fH$. The semigroup $\{T(t)\}_{t \geq 0}$ defined on $\cB(\fH)$ by $T(t)A = P^*(t)AP(t)$ is a basic example in the study of quantum dynamical semigroups, which are normally defined to be merely continuous for the ultraweak topology\cite{Fa99}.

\begin{proposition}
The semigroup $\{T(t)\}_{t \geq 0}$ is a SCLE semigroup for the strict topology.
\end{proposition}

It is of interest to see whether more quantum dynamical semigroups are in fact continuous for the strict topology. This, however, goes beyond the scope of this paper.

\begin{proof}
Fix $A \in \cB(\fH)$. The strong continuity of $\{P(t)\}_{t \geq 0}$ implies the operator strong* continuity of $t \mapsto T(t)A$. Therefore, the trajectory $t \mapsto T(t)A$ is locally bounded for the strong* topology, and hence, by the principle of uniform boundedness for the norm topology. As the strict topology coincides with the strong* topology on bounded sets\cite[Lemma II.2.5 and Theorem III.5.7]{Ta79} $t \mapsto T(t)A$ is continuous for the strict topology. The semigroup is locally equi-continuous by Lemma \ref{lemma:strongly_cont_implies_equicont}.
\end{proof}

\section{Appendix: Stochastic domination and the Chernoff bound} \label{section:appendix_orderings}

In this appendix, we recall the definition of stochastic domination\cite[Section IV.1]{Li92} and give a number of useful results. 

\begin{definition}
Suppose that we have two random variables $\eta_1$ and $\eta_2$ taking values in $\bR$. We say that $\eta_1$ stochastically dominates $\eta_2$, denoted by $\eta_1 \succeq \eta_2$ if for every $r \in \bR$ we have $\PR[\eta_1 > r] \geq \PR[\eta_2 > r]$.
\end{definition}

\begin{lemma} \label{lemma:appendix_stochastic_domination_equivalence}
For two real valued random variables $\eta_1,\eta_2$, we have that $\eta_1 \succeq \eta_2$ if and only if for every bounded and increasing function $\phi$, we have
$\bE[\phi(\eta_1)] \geq \bE[\phi(\eta_2)]$.
\end{lemma}

We say that a random variable $\eta$ is $\text{Poisson}(\gamma)$ distributed, $\gamma \geq 0$, denoted by $\eta \sim \text{Poisson}(\gamma)$ if $\PR[\eta = k] = \frac{\gamma^k}{k!} e^{-\gamma}$.

\begin{lemma} \label{lemma:appendix_domination_poisson}
If $\eta_1 \sim \text{Poisson}(\gamma_1)$ and $\eta_2 \sim \text{Poisson}(\gamma_2)$ and $\gamma_1 \geq \gamma_2$, then $\eta_1 \succeq \eta_2$.
\end{lemma}

Using the theory of couplings \cite[Section IV.2]{Li92}, a proof follows directly from the fact that if $\gamma_1 \geq \gamma_2$, then $\eta_1$ is in distribution equal to $\eta_2 + \zeta$, where $\zeta \sim \text{Poisson}(\gamma_1 - \gamma_2)$.

\smallskip

The next result, introduced by Chernoff\cite{Ch52}, is useful in the context of stochastic domination.

\begin{proposition} \label{proposition:Chernoff}
Let $X$ be a random variable on $\bR$ for which there exists $\theta_0 > 0 $, such that for $\theta < \theta_0$, the Laplace transform $\bE[e^{\theta X}]$ exists. Let $\{X_i\}_{i \geq 1}$ be independent and distributed as $X$. Then for $c \geq \bE[X]$, we have
\begin{equation*}
\PR\left[\frac{1}{n} \sum_{i=1}^n X_i > c \right] < \exp\left\{-n \inf_{0< \theta < \theta_0} \left\{c \theta - \log \bE[e^{\theta X}] \right\}\right\}.
\end{equation*}
\end{proposition}

We give a proof for completeness.

\begin{proof}
For all $0< \theta < \theta_0$, we have
\begin{align*}
\PR\left[\frac{1}{n} \sum_{i=1}^n X_i > c \right] & = \PR\left[e^{\theta \sum_{i=1}^n X_i} > e^{n\theta c} \right] \\
& < \exp \left\{- \left( n \theta c - \log \bE \left[e^{\theta \sum_{i=1}^n X_i} \right] \right) \right\},
\end{align*}
where we used Markov's inequality in line 2. As the $X_i$ are independent, $\log \bE \left[e^{\theta \sum_{i=1}^n X_i} \right] = n \log \bE \left[e^{\theta X} \right]$, which yields the final result.
\end{proof}

\textbf{Acknowledgement}
The author thanks Frank Redig, Ajit Iqbal Singh and Markus Haase for reading an early version of the manuscript and some valuable comments. Also, the author thanks an anonymous referee for some helpful suggestions that improved the exposition.

The author is supported by The Netherlands Organisation for Scientific Research (NWO), grant number 600.065.130.12N109.

\bibliographystyle{plain}

\end{document}